\documentclass[a4paper,11pt]{article}

\usepackage[utf8]{inputenc}
\usepackage{lmodern}
\usepackage[T1]{fontenc}

\usepackage{amsmath, amssymb, amsthm}
\usepackage{enumitem}
\usepackage{bbm,dsfont}
\usepackage{color}
\numberwithin{equation}{section}

\newcommand{\CP}{(\!\textit{CP})}
\DeclareMathOperator{\ospan}{span}

\newcommand{\1}{\mathds{1}}

\newcommand{\R}{\mathds{R}}
\newcommand{\C}{\mathds{C}}
\newcommand{\N}{\mathds{N}}
\newcommand{\A}{\mathcal{A}}

\newcommand{\D}{\mathcal{D}}

\renewcommand{\L}{\mathcal{L}}
\renewcommand{\d}{\mathrm{d}}
\renewcommand{\Re}{\operatorname{Re}}

\newcommand{\fra}{\mathfrak{a}}

\def\me{\mathsf{e}}
\def\mv{\mathsf{v}}
\def\diag{\mathrm{diag}}

\newcommand{\abs}[1]{\lvert #1 \rvert}
\newcommand\norm[1]{\lVert #1 \rVert} 
\newcommand\scp[2]{( #1, #2 )}
\newcommand\dual[2]{\langle #1, #2 \rangle}

\theoremstyle{plain}
\newtheorem{theorem}{Theorem}[section]
\newtheorem{proposition}[theorem]{Proposition}
\newtheorem{lemma}[theorem]{Lemma}
\newtheorem{corollary}[theorem]{Corollary}

\theoremstyle{definition}
\newtheorem{definition}[theorem]{Definition}

\theoremstyle{remark}
\newtheorem{remark}[theorem]{Remark}

\title{Diffusion in networks with time-dependent transmission conditions}
\author{
 	Wolfgang Arendt,
 	Dominik Dier,
	Marjeta Kramar Fijav\v{z}\thanks{Corresponding author}}

\begin{document}
\maketitle
\begin{abstract}
We study diffusion in a network which is governed by non-au\-tono\-mous Kirchhoff conditions at the vertices  of the graph. Also the diffusion coefficients may depend on time. We prove at first a result on existence and uniqueness using form methods. Our main results concern the long-term behavior of the solution. In the case when the conductivity and the diffusion coefficients match (so that mass is conserved) we show that the solution converges exponentially fast to an equilibrium. We also show convergence to a special solution in some other cases.

\smallskip

\end{abstract}

%%%%%%%%%%%%%%%%%%%%%%%%%%%%%%%%%%%%%%%%%%%%%%%%%%%%%%%%%%%%%%%%%%%%%%%%%%%%%%%%%%%%%%%%%%%%%%%%%%%%%%%%%%%%%%%%%%%%%%%%%%%%%%%%
\section{Introduction}
The aim of this paper is to study a non-autonomous dynamical system in a network subject to non-autonomous Kirchhoff knot-conditions at the vertices (knots). To be more precise we consider a finite, simple, connected graph $G$ with $m$ edges $\me _1,\dots,\me _m$ and $n$ vertices $\mv_1,\dots,\mv_n$. We study a diffusion system of the form 
\begin{equation}\label{diff-sys}
\frac{\partial u_j}{\partial t}(t,x)=c_j(t) \frac{\partial^2u_j}{\partial x^2}(t,x)+F_j(t,x),
\end{equation}
$j=1,\dots,m$, $t>0$, $x\in(0,1)$.
Here we think of $u_j(t,x)$ as a function defined on the $j$-th edge $\me_j$ and impose that, for each $t\ge 0$,  $\left(u_1(t,\cdot),\dots,u_m(t,\cdot)\right) $ is continuous on the graph, i.e.~that the values of $u_j(t,\cdot)$ are compatible at the vertices. We further impose Kirchhoff conditions at the vertices
\begin{equation}\label{diff-sys-boundary}
\sum_{j=1}^m \phi_{ij} \mu_j(t) \frac{\partial u_j}{\partial x}(t,\mv_i) = 0.
\end{equation}
Here $\Phi=(\phi_{ij})$ is the incidence matrix of the graph and $\mu_j(t)$ are given positive conductivity factors. Finally we impose an initial value condition
\begin{equation}
u(0,x)=u^0(x).
\end{equation}
We will show well-posedness (Theorem~\ref{network-wp}) and study the asymptotic behavior of the solution as time tends to infinity. 
For example, if $\sum_{j=1}^m \int_0^1 F_j(t) \, \d{t} = 0$ and $c_j = \mu_j$, then we show that the solution converges to an equilibrium (Theorem~\ref{thm:stability}).

Such evolutionary systems on networks have been studied by many authors, in particular in the autonomous case we refer to the monograph \cite{LLS94} and the proceedings \cite{ABN01}.
Concerning the non-autonomous case we mention the work by von Below \cite{Be88}. Here we use a different strategy than von Below, namely we use form methods to prove well-posedness. In the autonomous case this form method allowed in \cite{KMS07} to establish a holomorphic semigroup, and the asymptotic behavior of the solutions was studied with the help of spectral theory.
There are results on existence and uniqueness for non-autonomous forms, much developed by J.~L.~Lions, but the problem is to obtain solutions with sufficient regularity to identify the knot-conditions at the vertices.
This is indeed possible with the help of a recent result \cite{ADLO12} which we use here.
We need to assume that the conductivity factors $\mu_j(t)$ are Lipschitz continuous in time. But then we obtain a unique solution for each initial value $u^0\in V$ which is continuous from $\R_+$ with values in $V$ where 
\[V=\{f\in H^1(0,1)^m : f \text{ is continuous on the graph}\}.\]

For studying the long-term behavior of the solution we need at first to show that the eigenvalues depend continuously on time. This is done with the help of a criterion which is known to specialists but  is possibly nowhere  formulated explicitly. We give precise information in the Appendix.

This paper seems to be among the first studies of asymptotic behavior for non-autonomous evolution equation in a network.
A recent manuscript  \cite{BDK13} treats a first order transport equation on the edges of a network with time-varying transmission conditions at the vertices using the approach of difference equations and evolution families. However, the nature of our problem and the techniques we use are completely different.

 It is natural that the asymptotic behavior depends on conditions on the coefficients. We consider several cases. If the diffusion coefficients $c_j(t)$ and the conductivity coefficients $\mu_j(t)$ are equal, then mass is preserved. We show in this case that the solution converges exponentially fast to an equilibrium (Section \ref{section:stab1}). As a second case we assume that $b_j(t):=\frac{\mu_j(t)}{c_j(t)}$ satisfies a monotonicity condition (Section \ref{section:stab2}). Then the solution converges exponentially fast to a special solution. Finally, if $\dot{b}_j(t)\le c b_j(t)$ (Section \ref{section:stab3}), then we can still prove exponential convergence, but we have less information on the limit solution.

The paper is organized as follows. Section \ref{section:forms} is devoted to abstract non-autonomous forms. We give the well-posedness result (Theorem \ref{wp}) which we use later and prove continuity of eigenvalues as functions of $t$. In Section  \ref{section:diffusion_on_network} we explicitly describe our network diffusion problem and establish the well-posedness. 
In Section \ref{section:positivity} we investigate when the solutions are positive.
The qualitative behavior of solutions for $t \to \infty$ is studied in Section~\ref{section:stability}. Finally, we show in the Appendix that the $k$-th eigenvalue of a self-adjoint operator is continuous with respect to convergence in the resolvent sense.

%%%%%%%%%%%%%%%%%%%%%%%%%%%%%%%%%%%%%%%%%%%%%%%%%%%%%%%%%%%%%%%%%%%%%%%%%%%%%%%%%%%%%%%%%%%%%%%%%%%%%%%%%%%%%%%%%%%%%%%%%%%%%%%%
\section{Preliminaries: Forms and associated operators}\label{section:forms}

In this preliminary section we introduce the abstract framework. 
Let $V$ and $H$ be two separable real Hilbert spaces such that $V\underset d \hookrightarrow H$, i.e.\ $V$ is continuously and densely embedded in $H$.

We denote the scalar products and norms of $H$ and $V$ by $\scp{\cdot}{\cdot}_H$, $\norm{\cdot}_H$ and $\scp{\cdot}{\cdot}_V$, $\norm\cdot_V$, respectively.
Let $\fra \colon V \times V \to \R$ be a bilinear and \emph{continuous} mapping, i.e.
\begin{equation}\label{eq:a_continuous}
    \abs{\fra(u,v)} \le M \norm u _V \norm v _V \quad (u,v \in V)
\end{equation}
for some constant $M\ge 0$. We assume that $\fra$ is \emph{$H$-elliptic},
i.e.\ there exist constants $\alpha >0$ and $\omega \in \R$ such that
\begin{equation}\label{eq:H-elliptic}
    \Re \fra(u,u) + \omega \norm u_H^2 \ge \alpha \norm u _V^2 \quad (u \in V).
\end{equation}
If $\omega=0$, we say that  the form $\fra$ is  \emph{coercive}. The
operator $\A \in \L(V,V')$ \emph{associated with the form} $\fra$ is defined by
\[
\dual{\A u}{v} = \fra(u,v) \quad (u,v \in V).
\]
Here $V'$ denotes the dual space of $V$ and $\dual\cdot\cdot$ denotes the duality between $V'$ and $V$.
As usual, we identify $H$ with a dense subspace of $V'$ (associating every $f \in H$ with the linear form $v \mapsto ( f, v )_H$).
Then $V'$ is a Hilbert space for a suitable scalar product.

Seen as an unbounded operator on $V'$ with domain $D(\A) = V$ the
operator $- \A$ generates a holomorphic semigroup on $V'$. 
The semigroup is bounded on a sector if $\omega=0$, in which case $\A$ is an isomorphism. 
Denote by $A$ the part of $\A$ on $H$, i.e.
\begin{align*}
    D(A) := {}& \{ u\in V : \A u \in H \},\\
    A u := {}& \A u.
\end{align*}
Note that in this case we have
\begin{align*}
 D(A) &= \{u\in V : \exists h\in H \text{ s.t. } \fra(u,v) = (h,v)_H \text{ for all }v\in V\},\\
Au&=  h.
\end{align*}
Then the operator $-A$ generates a holomorphic $C_0$-semigroup on $H$ (the
restriction of the semigroup generated by $-\A$ to $H$). 
We call $A$ the operator \emph{associated with $\fra$ on $H$}.
Note that by the Lax-Milgram Theorem $(\lambda + A)\colon D(A) \to H$ is bijective and $(\lambda+A)^{-1} \in \L(H)$ for all $\lambda \ge \omega$.
For more information, see e.g.\ the monographs 
\cite[Chap.\ 1]{Ouh05} or \cite[Chap.\ 2]{Tan79}. 

We now consider a symmetric form $\fra$, i.e.\ we suppose  that $\fra(u,v)=\fra(v,u)$ for all $u,v \in V$. 
The following estimates follow from the $H$-ellipticity of $\fra$.
\begin{proposition}\label{proposition:resolvent_bounds}
	Assume that $\fra$ is a symmetric, continuous, and $H$-elliptic bilinear form. 
	Let $\A$ be the operator  associated with $\fra$ and let $\omega$, $\alpha$ be the constants from $\eqref{eq:H-elliptic}$.
	Then 
	\begin{enumerate}
		\item[{\rm (i)}] $\norm{(\omega+\A)^{-1}}_{\L(V',V)} \le \frac 1 \alpha$,
		\item[{\rm (ii)}] $\norm{(\omega+\A)^{-1}}_{\L(V',H)} \le \frac {c_H} \alpha$,
		\item[{\rm (iii)}] $\norm{(\omega+\A)^{-1}}_{\L(H,V)} \le \frac {c_{V'}} \alpha$,
	\end{enumerate}
	where $c_H$ and $c_{V'}$ are the operator norms of the embeddings $V \hookrightarrow H$ and $H \hookrightarrow V'$, respectively.
\end{proposition}

If the embedding of $V$ in $H$ is even compact, we write $V \underset c \hookrightarrow H$,
then the symmetric form $\fra$ can be diagonalized in the following way.

\begin{theorem}[Spectral representation of symmetric forms]\label{theorem:spectral_representation}
	Let $H$ and $V$ be infinite dimensional Hilbert spaces with $V \underset c \hookrightarrow H$. 
	Let $\fra \colon V \times V \to \R$ be a symmetric, continuous and $H$-elliptic bilinear form.
	Then there exist an orthonormal basis $\{ e_n : n \in \N \}$ of $H$ 
	and $\lambda_n \in \R$ with $\lambda_1 \le \lambda_2 \le \dots$, $\lim_{n \to \infty} \lambda_n = \infty$, such that
	\begin{align*}
		V &= \Big\{ u\in H : \sum_{n=1}^\infty \lambda_n \scp{u}{e_n}^2_H < \infty \Big\},\\
		\fra(u,v) &= \sum_{n=1}^\infty \lambda_n \scp{u}{e_n}_H \scp{v}{e_n}_H \quad (u,v \in V),\\
		\A u &= \sum_{n=1}^\infty \lambda_n \scp{u}{e_n}_H e_n \quad (u \in V),\\
		D(A) &= \Big\{ u\in H : \sum_{n=1}^\infty \lambda_n^2 \scp{u}{e_n}^2_H < \infty \Big\},\\
		Au &= \sum_{n=1}^\infty \lambda_n \scp{u}{e_n}_H e_n \quad (u \in D(A)).
	\end{align*}
\end{theorem}
\noindent 
This is a direct consequence of the spectral theorem for compact, self-adjoint operators (cf.\ \cite[Satz 4.40 and 4.43]{AU10}).
The orthonormal basis in Theorem~\ref{theorem:spectral_representation} is not unique, of course.
But if  $u\in D(A)$ is an eigenvector of $A$, i.e.\ $Au=\lambda u$ for some $\lambda \in\R$,
then $u \in \ospan \{ e_n : \lambda_n = \lambda \} = \ker(A-\lambda)$.
Thus $\lambda_n$ is in fact the $n$-th eigenvalue of $A$ counting multiplicity.

Next we extend our framework to non-autonomous problems. Let $V$, $H$ be Hilbert spaces with $V \underset d \hookrightarrow H$.
\begin{definition}\label{lip-form}
A \emph{Lipschitz continuous (non-autonomous) closed form} is
a function $\fra \colon [0, \tau] \times V\times V \to \R$, where
    \begin{enumerate}[label={\alph*)}]
        \item $\fra(t, \cdot, \cdot)\colon V\times V \to \R$ is bilinear for all $t \in [0,\tau]$,
        \item there exists $\dot M \ge 0$ with 
		\[
			\abs{\fra(t,u,v)- \fra(s,u,v)} \le \dot M \abs{t-s} \norm u_V \norm v_V \quad (t,s \in [0,\tau],\ u,v \in V),
		\]
        \item there exists $M \ge 0$ with 
		\[
			\abs{\fra(t,u,v)} \le M \norm u_V \norm v_V \quad (t \in [0,\tau],\ u,v \in V),
		\]
        \item there exist $\alpha > 0$ and $\omega \in \R$ with 
		\[
			\Re \fra(t,u,u) +\omega \norm u_H^2 \ge \alpha \norm u_V^2 \quad (t \in [0,\tau], u\in V).
		\]
    \end{enumerate}
\end{definition}	
For the remaining section let $\fra$ be a symmetric non-autonomous Lipschitz continuous closed form. For
$t \in [0,\tau]$ we denote by $\A(t)$ the operator associated  with the form $\fra(t,\cdot,\cdot)$ and by $A(t)$ the part of $\A(t)$ in $H$.
We will need the following result from \cite{ADLO12}. 
	
	\begin{theorem}{\cite[Theorem 5.1]{ADLO12}}\label{wp}
	Let $B\colon [0,\tau] \to \L(H)$ be strongly measurable such that
	\begin{equation*}
    		\beta \norm u_H^2 \le \Re ({B(t)u} , u)_H \le \beta^{-1} \norm u _H^2 \quad (t \in [0,\tau],\ u \in H)
	\end{equation*}
	for a constant $1 >\beta > 0$.
    Let ${u}^0 \in V$, $f \in L^2(0,\tau;H)$.
    Then there exists a unique
    \[
        u \in H^1(0,\tau;H)\cap L^2(0,\tau;V)
    \]
    satisfying
    \begin{align*}
        B(t) \dot u(t) + A(t)u(t)  &= f(t) \quad \text{a.e.}\\
                        u(0)    &={u}^0.
    \end{align*}
    Moreover, $u \in C([0,\tau];V)$ and
    \begin{equation*}%\label{eq:MR_estimate}
    	\norm u_{L^2(0,\tau, V)}^2 + \norm{\dot u}_{L^2(0,\tau, H)}^2 \le C \Big[ \norm{{u}^0}^2_V + \norm f_{L^2(0,\tau;H)}^2 \Big],
    \end{equation*}
    where the constant $C$ depends merely on $\beta, M, \alpha, \tau$ and $\dot M$.
\end{theorem}
The point of this theorem is that the solution $u$ is in $H^1(0,\tau;H)\cap L^2(0,\tau;V)$;
i.e.~it has maximal regularity.
This will be important in our context in order to make sure that the  solution satisfies the desired boundary conditions.
Moreover, it is remarkable that the solution is continuous with values in $V$. In our applications this implies in particular that the knot conditions are satisfied at each moment $t\ge0$.

Now we consider Theorem~\ref{theorem:spectral_representation} in the non-autonomous case.
Since the operators $A(t)$ do not commute in general, the spectral decomposition 
(i.e.\ the orthonormal basis and the eigenvalues) will depend on $t$.
We will show that the dependence of the eigenvalues is continuous in $t$.
More precisely we assume the following.
Let $\fra\colon [0,\tau]\times V\times V \to\R$ be a Lipschitz-continuous closed form, where $V \underset c \hookrightarrow H$,
and assume that $\fra$ is symmetric.
Consider the operator $A(t)$ on $H$ associated with $\fra(t,\cdot,\cdot)$ on $H$
and denote by $\lambda_n(t)$ the n-th eigenvalue of $A(t)$ (counting multiplicity).

\begin{theorem}\label{theorem:continuity_of_eigenvalues}
	The function $\lambda_n\colon [0, \tau] \to \R$ is continuous for every $n\in\N$.
\end{theorem}
\begin{proof}
	Let $t\in [0, \tau]$ and consider the constant $\omega$ of Definition~\ref{lip-form} d).
	Then $\omega+A(t) \colon D(A(t)) \to H$ is bijective and $(\omega + A(t))^{-1} \in\L(H)$ with
	$(\omega + A(t))^{-1}H = D(A(t)) \subset V$.
	The mapping $R\colon H \to V$ given by $Ru=(\omega + A(t))^{-1}u$ is bounded by Proposition~\ref{proposition:resolvent_bounds}.
	Since the embedding $\iota\colon V\to H$ is compact by hypothesis, $(\omega + A(t))^{-1}= \iota \circ R(t) \in \L(H)$ is compact.

	 Let $s,t \in [0,\tau]$. Then
	\begin{align*}
		&\norm{(\omega+A(t))^{-1}-(\omega+A(s))^{-1}}_{\L(H)}\\ 
			&\quad= \norm{(\omega+\A(t))^{-1}(\omega+\A(s) - (\omega+\A(t)) )(\omega+\A(s))^{-1}}_{\L(H)}\\
			&\quad= \norm{(\omega+\A(t))^{-1}(\A(s) -\A(t) )(\omega+\A(s))^{-1}}_{\L(H)}\\
			&\quad\le \norm{(\omega+\A(t))^{-1}}_{\L(V',H)} \norm{\A(s) - \A(t) }_{\L(V,V')} \norm{(\omega+\A(s))^{-1}}_{\L(H,V)}\\
			&\quad\le \frac{c_{V'}} \alpha \dot M \abs{t-s} \frac {c_H}\alpha,
	\end{align*}
	by Proposition~\ref{proposition:resolvent_bounds} (ii), (iii) and property b) of the form $\fra$ in Definition~\ref{lip-form}.
	Thus the mapping $(\omega+A(\cdot))^{-1}\colon [0,\tau] \to \L(H)$ is continuous.
	Now the theorem follows by the result in the Appendix (Theorem~\ref{theorem:convergence_of_eigenvalues}).
\end{proof}

\begin{remark}\label{remark:no_Lip}
Note that  we could prove the continuity of the mapping $(\omega+A(\cdot))^{-1}$ above without using property b) of the form $\fra$ in Definition~\ref{lip-form}. We merely need that $\A(\cdot)\colon [0,\infty) \to \L(V,V')$ is continuous instead of Lipschitz continuous.
\end{remark}
%%%%%%%%%%%%%%%%%%%%%%%%%%%%%%%%%%%%%%%%%%%%%%%%%%%%%%%%%%%%%%%%%%%%%%%%%%%%%%%%%%%%%%%%%%%%%%%%%%%%%%%%%%%%%%%%%%%%%%%%%%%%%%%%
\section{Diffusion in networks}\label{section:diffusion_on_network}

We apply the theory from the previous section to treat non-autonomous diffusion processes in a finite network. The network is represented by a finite, simple, connected graph  $G$ with $m$ edges $\me _1,\dots,\me _m$ and $n$ vertices $\mv_1,\dots,\mv_n$. We identify all edges with the interval $[0,1]$. We further assume that all the vertices have degree at least 2, i.e., that each vertex is
incident to at least 2 edges. 

The structure of the network  can be described by various graph matrices. 
Here we use the $n\times m$ {\em incidence matrices} 
$\Phi^+:=(\phi^+_{ij})$,  $\Phi^-:=(\phi^-_{ij})$, and    $\Phi:=(\phi_{ij})$,  defined by
\begin{equation*}
\phi^+_{ij}:=\left\{
\begin{array}{rl}
1, & \hbox{if }\mv _i= \me _j(1),\\
0, & \hbox{otherwise},
\end{array}
\right.
\qquad
\phi^-_{ij}:=\left\{
\begin{array}{rl}
1, & \hbox{if } \mv _i=\me _j(0),\\
0, & \hbox{otherwise,}
\end{array}
\right.
\end{equation*}
and
 $$\Phi:=\Phi^+-\Phi^-.$$  Further,
let $\Gamma(\mv_i)$ be the set of all the indices of the edges
having an initial or  endpoint at $\mv _i$, i.e.,
\[
	\Gamma(\mv _i):=\left\{j\in \{1,\ldots,m\}: \phi_{ij}\ne 0 \right\}, \quad i=1,\dots,n.
\]
We will consider functions which are defined on the edges of the graph.
These functions will be assumed to be continuous on the graph.
By this we mean the following.
Let $f \in C[0,1]^m$, $f=(f_1,\dots f_m)$.
For $j \in \{ 1,\dots, m \}$ we set
$f_j(\mv_i) := f_j(0)$ if $\me_j(0) = \mv_i$ and $f_j(\mv_i) := f_j(1)$ if $\me_j(1) = \mv_i$.
Then we say that $f$ is \emph{continuous on the graph}, if
\begin{equation}\label{eq:continuous_on_graph} 
	f_j(\mv_i) = f_l(\mv_i) \quad (j,l \in \Gamma(\mv_i))
\end{equation}
for all $i=1, \dots, n$.

\begin{lemma}\label{lemma:continuous_on_graph}
A function $f \in C[0,1]^m$ is continuous on the graph if and only if there exists a vector $d\in {\C}^n$ such that
  $(\Phi^-)^\top d=f(0)$  and $(\Phi^+)^\top d=f(1)$.
\end{lemma}

\begin{proof}
Let  $f \in C[0,1]^m$ be a continuous function on the graph. For a vertex $\mv_i$ denote by $d_i$ the ``value of $f$ at $\mv_i$'', i.e.~$d_i:= f_j(\mv_i)$ for any $j\in\Gamma(\mv_i)$. Observe that for $d:=(d_i)_{i=1,\dots, n}$ we obtain $(\Phi^-)^\top d=f(0)$  and $(\Phi^+)^\top d=f(1)$.

Conversely, if   $(\Phi^-)^\top d=f(0)$  and $(\Phi^+)^\top d=f(1)$ for some $d\in {\C}^n$, 
then for every $j\in \Gamma(\mv_i)$  we have either $d_i=\phi_{ij}^- d_i = f_j(0)$ or $d_i=\phi_{ij}^+ d_i = f_j(1)$, so \eqref{eq:continuous_on_graph} holds. 
\end{proof}

We consider the following diffusion process in the network. At first we consider a finite time interval $[0,\tau]$ where $\tau>0$.
\begin{itemize}
\item On every edge $\me_j$, $j=1,\dots,m$,  evolution is governed by the heat equation:
\begin{equation}\label{diff_edge}
\dot{u}_j(t,x)= c_j(t) \cdot u_j''(t,x) \quad (t\in [0,\tau], \; x\in(0,1)),
\end{equation}
where $c_1(t),\ldots,c_m(t)$ are time-dependent positive diffusion coefficients.
Here the dot refers to the derivative with respect to the time, 
and the prime refers to the derivative with respect to the space variable.
\item We impose that for each $t>0$ the function $u(t,\cdot)=(u_1(t,\cdot),\dots,u_m(t,\cdot))$ is continuous on the graph; 
i.e., $\eqref{eq:continuous_on_graph}$ is satisfied.
\item Boundary conditions at the vertices $\mv_i$, $i=1,\ldots,n$, are Kirchhoff-type conditions of the form:
\begin{equation}\label{bound}
\sum_{j=1}^m\phi_{ij}\mu_j(t) u'_j(t,\mv_i)= 0 \quad (t\in [0,\tau]),\end{equation}
where  $\mu_1(t),\ldots,\mu_m(t)$ are given time-dependent positive conductivity factors.
Here we denote by $u_j'(t,\mv_i)$ the left derivative of the function $u_j(t,\cdot)$ at $0$ if $e_j(0)=\mv_i$ 
and the right derivative at $1$ if $e_j(1)=\mv_i$.
\item The initial conditions on the edges $\me_j$, $j=1,\dots,m$, are
\begin{equation}\label{ic}
u_j(0,x)={u}^0_{j}(x) \quad (x\in (0,1)). 
\end{equation}
\end{itemize}

Before identifying the problem with an abstract Cauchy problem  we write the continuity and boundary conditions in a more compact form. To this end we  introduce \emph{weighted incidence matrices}
 $\Phi^+_w(t):=\left(\omega^+_{ij}(t)\right)$ and $\Phi^-_w(t):= \left(\omega^-_{ij}(t)\right)$ with entries
\begin{equation*}
\omega^+_{ij}(t):=\left\{
\begin{array}{ll}
\mu_j(t) , & \hbox{if }\mv _i= \me _j(1),\\
0, & \hbox{otherwise},
\end{array}
\right. \hbox{ and}\quad \omega^-_{ij}(t) :=\left\{
\begin{array}{ll}
\mu_j(t), & \hbox{if } \mv _i=\me _j(0),\\
0, & \hbox{otherwise}.
\end{array}
\right.
\end{equation*}
With these notations,  the Kirchhoff law $\eqref{bound}$ can be written in matrix form as
\begin{equation}
\label{kl}
\Phi_w^-(t)\cdot u'(t,0)=\Phi_w^+(t)\cdot u'(t,1) \quad (t\in [0,\tau]).
\end{equation}

In order to use the previously introduced form methods we now first define the appropriate Hilbert spaces
\[
	H := L^2(0,1)^m\text{ and } V:=\left\{f\in H^1(0,1)^m : f \text{ is continuous on the graph} \right\}
\]
equipped with the inner products
\begin{align}
  (f,g)_H &= \sum_{j=1}^m \int_0^1f_j(x) g_j(x) \ \d x \quad (f,g\in H) \label{inn_H}\\
(f,g)_V&= \sum_{j=1}^m \int_0^1\left(f_j(x) g_j(x) +f'_j(x) g'_j(x)\right) \; \d x \quad (f,g\in V).\label{inn_V}
\end{align}
Recall that $H^1(0,1) \hookrightarrow C[0,1]$ and so $V$ is a closed subspace of $H^1(0,1)^m$.
Note that $V$ is a dense subspace of $H$.
For $t\in [0,\tau]$ we define the function $\fra(t,\cdot,\cdot) \colon V\times V \to \R$ by
 \begin{equation}\label{form}
 \fra(t,f,g)=\sum_{j=1}^m \mu_j(t)\int_0^1f'_j(x) g'_j(x) \ \d x.
 \end{equation}
Since $\mu_j(t) > 0$ for all $j=1,\dots,m$ one easily gets that $\fra(t,\cdot,\cdot)$ is a symmetric, continuous, $H$-elliptic bilinear form.

\begin{proposition}\label{asoc-op}
Let $t >0$. The operator associated with the form $\fra(t,\cdot,\cdot)$ on $L^2(0,1)^m$ is given by
\begin{align*}
A(t)&:= \diag\left(- \mu_j(t) \frac{\d^2}{\d x^2}\right)_{j=1,\dots,m},\\
D(A(t))&:=\left\{u\in V\cap \left(H^2(0,1)\right)^m :  \Phi_{w}^-(t)\cdot u'(0)=\Phi_{w}^+(t)\cdot u'(1)\right\}.
\end{align*}
 \end{proposition}

\begin{proof}
Remember that the  operator $\widetilde{A}(t)$ associated with the form $\fra(t,\cdot,\cdot)$ on $H$ is defined by
\begin{align*}
D(\widetilde A(t)) &= \{u\in V : \exists h\in H \text{ s.t. } \fra(t,u,v) = (h,v)_H \text{ for all }v\in V\},\\
\widetilde A(t) u&=  h.
\end{align*}
We show that $A(t)=\widetilde{A}(t)$.

First we prove $A(t)\subset \widetilde{A}(t)$.
Let $u \in D(A(t))$ and $h:= A(t) u$. 
Since $u \in V\cap \left(H^2(0,1)\right)^m$, integration by parts yields
\begin{align*}
\fra(t,u,v)&=\sum_{j=1}^m \mu_j(t)\int_0^1u'_j(x) v'_j(x) \; \d x\\
&=\sum_{j=1}^m \mu_j(t)\left[u'_j(1)v_j(1)-u'_j(0)v_j(0) \right] -\sum_{j=1}^m \mu_j(t) \int_0^1u''_j(x) v_j(x) \; \d x
\end{align*}
for all $v\in V$. By Lemma~\ref{lemma:continuous_on_graph}, for every $v\in V$ there exists $d_v\in\C^n$ such that $v(0)=(\Phi^-)^\top d_v$ and $v(1)=(\Phi^+)^\top d_v $. 
Note also  that the incidence matrices $\Phi^-$ and $\Phi^+$ have exactly one nonzero entry in every column (corresponding to the two ends of an edge) and have the same zero-pattern as the  weighted incidence matrices $\Phi_{w}^-(t)$ and $\Phi_{w}^+(t)$, respectively. Thus
\begin{equation}\label{eq:boundary_cond}
\sum_{j=1}^m \mu_j(t)\left[u'_j(1)v_j(1)-u'_j(0)v_j(0) \right]  = d_v^\top\cdot \left( \Phi_{w}^+(t) \cdot u'(1)-\Phi_{w}^-(t)\cdot u'(0)  \right) = 0
\end{equation}
for all $v\in V$. 
Hence $\fra(t, u, v) = (h,v)_H$ for all $v \in V$.

To show that $\widetilde{A}(t)\subset A(t)$, let $u \in D(\widetilde{A}(t))$ and set $h:= \widetilde{A}(t)u$.
Since $\fra(t,u,v) =(h,v)_H$ for all $v \in V$ and $\D(0,1)^m \subset V$ 
we obtain by  integration by parts as above that $u \in H^2(0,1)^m$ and $h_j = - \mu_j(t) u''_j$.
Thus 
\begin{equation*}
	\sum_{j=1}^m \mu_j(t)\left[u'_j(1)v_j(1)-u'_j(0)v_j(0) \right] = 0
\end{equation*}
for all $v \in V$.
Using Lemma~\ref{lemma:continuous_on_graph} as before we obtain \eqref{eq:boundary_cond} for every $v\in V$. The entries of $(d_v)_i$ are the joint values of $v_j$ attained at vertex $\mv_i$ for $j\in\Gamma(\mv_i)$. Since this holds for arbitrary $v\in V$, it follows
\[\Phi_\omega^-\cdot u'(0) = \Phi_\omega^+\cdot u'(1).\]
Hence  $u \in D(A(t))$ and $A(t)u = h$.
\end{proof}

Note that the domain $D(A(t))$ consists of continuous  functions on the graph that satisfy the non-autonomous Kirchhoff-knot conditions $\eqref{bound}$. Thus in view of Proposition~\ref{asoc-op} we may rewrite $\eqref{diff_edge}$-$\eqref{ic}$ in the form
\begin{align*}
        \dot u(t) + A(t)u(t)  &= 0 \quad \text{a.e.}\\
                        u(0)    &={u}^0, 
\end{align*}
if $\mu_j(t)=c_j(t)$ for all $j=1,\dots,m$.
In order to treat the case $\mu_j \neq c_j$ we introduce the operator 
\begin{equation}\label{eq:op-B}
B(t):= \diag(b_j(t)) \in \L(H)\text{ with }b_j(t):=\frac{\mu_j(t)}{c_j(t)}.
\end{equation}
Then $\eqref{diff_edge}$-$\eqref{ic}$ can be written as
\begin{equation}\begin{aligned}\label{eq:cp}
        B(t)\dot u(t) + A(t)u(t)  &= 0 \quad (\text{a.e. }t\in[0,\tau]  )\\
                        u(0)    &={u}^0.
\end{aligned}\end{equation}
Assuming some regularity properties on the time-depending coefficients $c_j(t)$ and $\mu_j(t)$ appearing in the equations \eqref{diff_edge} and the boundary conditions \eqref{bound} we finally obtain the following well-posedness result of problem $\eqref{diff_edge}$-$\eqref{ic}$.

\begin{theorem}\label{network-wp}
Let the functions $\mu_j\colon [0,\tau] \to\R^+$ be Lipschitz continuous and the functions $c_j\colon [0,\tau] \to\R^+$ be measurable.
Suppose there exists an $\varepsilon\in (0,1)$ such that 
\[
	\varepsilon \le\mu_j(t), c_j(t)\le\varepsilon^{-1} \quad (j=1,\dots m, \, t\in[0,\tau]).
\]
Let  $u^0\in V$ and $F\in L^2(0,\tau;H)$ be given. Then there exists a unique function
\[
	u \in H^1(0,\tau;H) \cap C([0,\tau];V)
\]
such that for all  $j=1,\dots, m$ the following holds

\begin{align*}
       & \dot{u}_j(t,x)=c_j u_j''(t,x)+F_j(t,x)  &(t\in[0,\tau],\, x\in (0,1)),\\[2mm]
     &  \sum_{k=1}^m\phi_{ik}\mu_k(t) u'_k(t,\mv_i)=0  &(t\in[0,\tau], \, i=1,\dots, n),\\[2mm]
       &   u_j(0,x)    =u^0_j(x).&
\end{align*}

\end{theorem}
\begin{proof}
We check the conditions of Theorem~\ref{wp}.
First we have to show that the function $\fra \colon [0,\tau]\times V\times V \to \R$ defined by $\eqref{form}$ is a Lipschitz continuous closed form.
This is true for $M := \max_{t\in[0,\tau],j=1,\dots,m} \mu_j(t)$, $\dot M := \max_{j=1,\dots,m} L_j$ where $L_j$ are the Lipschitz constants of the fuctions $\mu_j$
and $\omega,\alpha := \varepsilon$.

The function $B\colon [0,\tau]\to \L(H)$ is measurable since the functions $b_j(\cdot):=\frac{\mu_j(\cdot)}{c_j(\cdot)}$ are measurable.
Moreover note that  $\varepsilon^2 \le b_j(t) \le \varepsilon^{-2}$ for all $t\in[0,\tau]$, thus we obtain that
\begin{equation}\label{eq:estimate-B}
	\varepsilon^2 \norm{g}_H^2  \le \Re (B(t)g , g)_H \le\varepsilon^{-2}\norm{g}_H^2 
\end{equation}
for all $g\in H$ and $t\in[0,\tau]$. Now given $F\in L^2(0,\tau; H)$ we define $\widetilde{F}(t):= B(t) F(t).$ Then by  Theorem~\ref{wp} we find a unique solution $u$ of
\[B(t)\dot{u}(t) + A(t)u(t)=\widetilde{F}(t)\quad (\text{a.e. }  t \in [0,\tau]),\quad u(0)=u^0.
\]
Dividing by $b_j(t)$ we see that $u$ is a solution of the system as formulated in the Theorem. 
\end{proof}

Thus we have well-posedness for every choice of the diffusion coefficients $c_j(t)$ and the conductivity factors $\mu_j(t)$ (up to slight regularity assumptions).

Our aim is to examine stability. Therefore we  extend Theorem~\ref{network-wp} to the half line $[0,\infty)$.
This is possible since $\tau > 0$ was arbitrary. 
We formulate the result using an abstract (but equivalent) notation, which will be more convenient in Section~\ref{section:stability}.

\begin{corollary}\label{corollary:network_wp}
	Let the functions $\mu_j\colon [0,\infty) \to\R^+$ be locally Lipschitz continuous and the functions $c_j\colon [0,\infty) \to\R^+$ be measurable.
	Suppose there exists an $\varepsilon\in (0,1)$ such that 
	\[
		\varepsilon \le\mu_j(t), c_j(t)\le\varepsilon^{-1} \quad (j=1,\dots m, \, t\in[0,\infty)).
	\]
	Then for every $u^0\in V$ and $F \in L^2_{loc}([0,\infty);H)$ there exists a unique solution 
        	$u \in H^1_{loc}([0,\infty);H)\cap L^2_{loc}([0,\infty);V)\cap C([0,\infty),V)$
	of the non-autonomous Cauchy problem
\begin{equation*}
		\begin{aligned}
			B(t)\dot u(t) + A(t)u(t) &= F(t) \quad (\text{a.e. } t \in [0,\infty) )\\
				u(0) &= u^0;
		\end{aligned}
\end{equation*}
i.e.\ $u\vert_{[0,\tau]} \in L^2(0,\tau;V) \cap H^1(0,\tau;H)\cap C([0,\tau],V)$ is the unique solution of
\[
	B(t)\dot u(t) + A(t)u(t) = F(t) \quad (\text{a.e. } t \in [0,\tau]),\quad u(0)=u^0
\]
for all $\tau > 0$.
\end{corollary}

%%%%%%%%%%%%%%%%%%%%%%%%%%%%%%%%%%%%%%%%%%%%%%%%%%%%%%%%%%%%
\section{Positivity}\label{section:positivity}

In this section we show that the solution in Theorem~\ref{network-wp} is positive
whenever the initial value and the inhomogeneity $F$ are positive.
We start by a result which is of independent interest.

Let $(\Omega,\Sigma,\mu)$ be a measure space.
Then $f \in L^2(\Omega)$ may be decomposed in its positive and negative parts
$f=f^+ - f^-$ where
$f^+(x)=f(x)$ if $f(x) > 0$ and $f^+(x)=0$ otherwise, while $f^-=(-f)^+$.

\begin{proposition}\label{prop:positivity}
Let $u\in H^1(0,\tau;L^2(\Omega))$. Define $u^+\colon [0,\tau]\to L^2(\Omega)$ by
\[
u^+(t):=u(t)^+ \quad (t\in[0,\tau]).
\]
Then $u^+ \in H^1(0,\tau;L^2(\Omega))$ and
\begin{equation}\label{eq:u+'}
(u^+)\dot{}(t)=\dot u(t) \cdot \1_{\{u(t)>0\}} \quad (\text{a.e. }t\in[0,\tau]).
\end{equation}
Here $\1_{\{u(t)>0\}}$ is the characteristic function of the set   $\{x \in \Omega : u(t)(x) > 0\}$.
\end{proposition}
\begin{proof}
The mapping $f\mapsto f^+ \colon L^2(\Omega)\to L^2(\Omega)$ is Lipschitz continuous.
It follows from \cite[Theorem 5.3]{ADO13} that $u^+ \in H^1(0,\tau;L^2(\Omega))$.
In particular, $u$ and $u^+$ are differentiable outside a null set $M \subset [0,\tau]$.
We assume that $0,\tau \in M$. Let $t\in[0,\tau]\setminus M$.
We show that  \eqref{eq:u+'} holds for $t\in[0,\tau]\setminus M$.
Let $h_n \downarrow 0$ as $n\to \infty$.
Passing to a subsequence if necessary we find a null set $N\subset \Omega$ such that
\begin{align*}
\tfrac1{\pm h_n}(u(t\pm h_n)(x)-u(t)(x)) &\to \dot u(t)(x)\\
\tfrac1{\pm h_n}(u^+(t\pm h_n)(x)-u^+(t)(x)) &\to (u^+)\dot{}(t)(x)\\
	u(t\pm h_n)(x) &\to u(t)(x)
\end{align*}
as $n\to\infty$ for all $x\in\Omega\setminus N$.
Now let $x\in\Omega\setminus N$ and consider three cases.
\begin{description}
\item{1st case: $u(t)(x)>0$.}
Then $u(t\pm h_n)(x)>0$ for $n$ sufficiently large.
Thus
\[
(u^+)\dot{}(t)(x) = \lim_{n\to\infty} \tfrac1{\pm h_n}(u^+(t\pm h_n)(x)-u^+(t)(x)) = \dot u(t)(x).
\]
\item{2nd case: $u(t)(x) <0$.}
Then $(u^+)\dot{}(t)(x) = 0$ by the same argument.
\item{3rd case: $u(t)(x)=0$.} Then
\[
(u^+)\dot{}(t)(x) = \lim_{n\to\infty} \tfrac1{\pm h_n}(u^+(t\pm h_n)(x)) = 0
\]
since $h_n>0$ and $-h_n <0$.
\end{description}
\end{proof}

Proposition~\ref{prop:positivity} is somehow analogous to the classical result,
valid for each open set $\Omega \subset \R^d$:
If $v\in H^1(\Omega)$, then $v^+ \in H^1(\Omega)$ and
\[D_j v^+ = \1_{\{ v>0 \}} D_j v \quad (j=1,\dots,d).\]

Now we can prove the result on positive solutions.
We consider the situation of Theorem~\ref{network-wp}.
Let $V$ and $H$ be the Hilbert spaces defined in Section~\ref{section:diffusion_on_network}.
For $v \in H$ we define 
\[
v^+:=(v_j^+)_{j=1,\dots, m}
\] 
and write $v \ge 0$ if $v = v^+$.
Further we set $v^-:=(-v)^+$, then $v=v^+-v^-$ and $v^+, v^- \ge 0$.
Note that if $v \in V$, then also $v^+, v^- \in V$ and 
\begin{equation}\label{eq:positivity}
(v_j^+)'(v_j^-)'=0 \quad (j=1,\dots,m).
\end{equation}
Finally we denote the positive cones by
\[
H_+:=\{u\in H:u\ge0\} \text{ and } V_+:=\{v\in V:v\ge0\}.
\]
\begin{theorem}\label{thm:positivity}
Assume that all the conditions of Theorem~\ref{network-wp} hold and additionally suppose that the functions $b_i$ are Lipschitz continuous, i.e.\ $b_i \in W^{1,\infty}(0,\infty)$. Moreover, let $F \in L^2(0,\tau;H_+)$ and $u^0 \in V_+$.
Then the solution $u$ provided by Theorem~\ref{network-wp} is positive, i.e.\ it takes values in $V_+$.
\end{theorem}
\begin{proof}
Let $c :=\frac{1}{2\varepsilon^2} \max_i \norm{\dot b_i}_\infty$, where $\varepsilon>0$ is as in Theorem~\ref{network-wp}.
We set $v(\cdot):= e^{-c \cdot} u(\cdot)$ and $\tilde F(\cdot) = e^{- c \cdot} F(\cdot)$.
Then $F$ is still positive and, since $u$ is a solution of \eqref{eq:cp} we have
	\begin{align*}
		B(t) \dot v(t) &= e^{-c t} (B(t) \dot u(t)-c B(t) u(t))\\
			&= \tilde F(t) - A(t) v(t) - c B(t) v(t).
	\end{align*}
Observe that $v^-(0)=0$ since $u(0)\ge 0$.
Using \eqref{eq:estimate-B} and Proposition~\ref{prop:positivity} we obtain the following estimate for each $t \in [0,\tau]$:
\begin{align*}
	\varepsilon^2 \norm{v^-(t)}_H^2
	&\le \norm{B^{1/2}(t) v^-(t)}_H^2
	= \norm{B^{1/2}(t) v^-(t)}_H^2 - \norm{B^{1/2}(0) v^-(0)}_H^2\\
	&= -2 \int_0^t  \scp{B(s)\dot v(s)}{v^-(s)}_H \ \d s + \int_0^t \scp{\dot B(s)v^-(s)}{v^-(s)}_H \ \d s\\
	&= -2 \int_0^t \scp{\tilde F(s)- A(s) v(s)}{v^-(s)}_H\ \d s \\
	&\quad+ \int_0^t \scp{(\dot B(s) - 2cB(s)) v^-(s)}{v^-(s)}_H\ \d s\\
	&\le -2 \int_0^t \scp{\tilde F(s)}{v^-(s)}_H \ \d s + 2 \int_0^t \fra(s, v(s), v^-(s)) \ \d s\\
	&\le 2 \int_0^t \fra(s, v(s), v^-(s))  \ \d s- 2\int_0^t \fra(s, v^+(s), v^-(s)) \ \d s\\
	&\le -2 \int_0^t \fra(s, v^-(s), v^-(s)) \ \d s \le 0.
	\end{align*}
Here we used that $\fra(s,v^+,v^-)=0$ for all $v\in V$ by \eqref{eq:positivity}.
This shows that $v^-(t)=e^{-ct}u^-(t) = 0$, hence $u(t) \ge 0$ for every $t \in [0,\tau]$.
\end{proof}

%%%%%%%%%%%%%%%%%%%%%%%%%%%%%%%%%%%%%%%%%%%%%%%%%%%%%%%%%%%%%%%%%%%%%%%%%%%%%%%%%%%%%%%%%%%%%%%%%%%%%%%%%%%%%%%%%%%%%%%%%%%%%%%%
\section{Stability}\label{section:stability}

Let $V$ and $H$ be the Hilbert spaces defined in Section~\ref{section:diffusion_on_network}.
Note that $V \underset c \hookrightarrow H$ since $H^1(0,1) \underset c \hookrightarrow L^2(0,1)$ (see \cite{Bre11}).
Recall also the definition of the operator families $(A(t))_{t\in [0,\infty)}$ and $(B(t))_{t\in [0,\infty)}$ from Section~\ref{section:diffusion_on_network}.
We will assume all the conditions of  Corollary~\ref{corollary:network_wp}. Note that then there exists a $\beta\in (0,1)$ such that
\begin{equation}\label{eq:B_bounds}
 \beta \norm{g}_H \le \norm{B^{1/2}(t) g}_H \le \beta^{-1} \norm{g}_H
\end{equation}
 holds for for all $g\in H$ and all $t \in [0,\infty)$, cf.~\eqref{eq:estimate-B}.
Given some $F \in L_{loc}^2([0,\infty);H)$ and $u^0\in V$ we denote by $u$ the solution to the abstract Cauchy problem 
\begin{equation*}
	\CP \left\{ 
		\begin{aligned}
			B(t)\dot u(t) + A(t)u(t) &= F(t) \quad (\text{a.e. } t \in [0,\infty) )\\
				u(0) &= u^0,
		\end{aligned}
	\right.
\end{equation*}
given by Corollary~\ref{corollary:network_wp}.
In this section we discuss the asymptotic behavior of $u$ for several cases of the operator family $(B(t))_{t\in [0,\infty)}$.

Since $V \underset c \hookrightarrow H$ we can apply Theorem~\ref{theorem:spectral_representation} to the form $\fra(t,\cdot,\cdot)$ for any $t \in[0,\infty)$.
Thus the spectrum of the operator $A(t)$  for any $t \in[0,\infty)$ consists of an increasing sequence of eigenvalues $\lambda_1(t)\le\lambda_2(t)\le\cdots$ with the corresponding orthonormal basis $\{e_n(t): n\in\N\}$ of eigenvectors.
The following properties of the first and second eigenvalue and the first eigenvector of $A(t)$ are essential for the analysis of the asymptotic behavior of the solution $u$.
\begin{lemma}
	Let $t\in [0,\infty)$. 
	The normalized  first eigenvector of $A(t)$ equals  $e_1(t) \equiv\frac 1 {\sqrt m} \1$ and the first eigenvalue $\lambda_1(t)\equiv0$.
	The second eigenvalue $\lambda_2(t)$ is positive.
\end{lemma}
\begin{proof}
	Note that $(A(t)u,u)_H = \fra(t,u,u) \ge 0$ for all $t\in[0,\infty)$ and all $u \in D(A(t))$, thus the first eigenvalue is non-negative by Theorem~\ref{theorem:spectral_representation}.
	Observe that $ e_1(t)\in D(A(t))$ and $A(t)e_1(t)=0$ for all $t \in [0,\infty)$, which shows the assertions on $e_1(t)$ and $\lambda_1(t)$.
	
	Since $(e_2(t),e_1(t))_H=0$ we have that $e_2(t)$ is not constant, i.e.\ $\norm{e_2'(t)}_H \neq 0$. 
Thus $\lambda_2(t)=(A(t)e_2(t),e_2(t))_H=\fra(t,e_2(t),e_2(t))>0$ by 
\eqref{form}. 
\end{proof}
Since the eigenvector $e_1(t)$ does not depend on $t$ we set 
$$e_1 := \frac 1 {\sqrt m} \1.$$
We can improve the result on pointwise positivity of $\lambda_2(\cdot)$.
\begin{proposition}
	There exists a constant $\underline \lambda_2 > 0$ such that $\lambda_2(t) \ge \underline \lambda_2$ for all $t \in [0, \infty)$.
\end{proposition}
\begin{proof}
	By Theorem~\ref{theorem:continuity_of_eigenvalues} the function $\lambda_2\colon [0,\infty) \to [0,\infty)$ is continuous.
	Suppose that there exists a sequence $(t_n)_{n\in\N} \subset [0,\infty)$ such that $\lambda_2(t_n) \to 0$.
	After taking a subsequence we may assume that $t_n \to t \in [0,\infty]$ as $n \to \infty$.
	If $t$ is finite $\lambda_2(t_n) \to \lambda_2(t) > 0$ as $n \to \infty$ which is a contradiction.
	
	For the case $t = \infty$ we also may assume that $t_n \uparrow \infty$ (after taking a subsequence). 
	Recall the definition of the form $\fra$ in \eqref{form}. 
	Since the sequences $(\mu_j(t_n))_{n\in\N}$ are bounded we may assume (after taking a subsequence) that they converge.
	Define continuous functions $\tilde \mu_j\colon [0, t_1^{-1}]\to [\varepsilon,\varepsilon^{-1}]$ with $\tilde \mu_j(t_n^{-1})=\mu_j(t_n)$ 
	where $\varepsilon > 0$ is chosen as in Corollary~\ref{corollary:network_wp}. 
	For example define $\tilde \mu_j$ affine on the intervals $[t_{n}^{-1}, t_{n-1}^{-1}]$ and set $\tilde \mu_j(0)=\lim_{n\to \infty} \mu_j(t_n)$, 
	then these functions are continuous on $(0,t_1^{-1}]$ by definition and continuous in $0$ by a sandwich argument.
	For these functions we define the form $\tilde\fra$ analoguos to $\fra$. 
	We denote by $\tilde\A(t)$ the associated operator with the form $\tilde\fra(t,\cdot,\cdot)$ and by $\tilde A(t)$ its part on $H$. 
	The corresponding second eigenvalue of $\tilde A(t)$ is denoted by $\tilde \lambda_2(t)$. 
	Now observe that $\tilde\fra(t_n^{-1},u,v)=\fra(t_n, u,v)$ 
	and therefore $\tilde\lambda_2(t_n^{-1})=\lambda_2(t_n)$, hence $\tilde\lambda_2(t_n^{-1}) = \lambda_2(t_n) \to 0.$
	On the other hand continuity of $\tilde\lambda_2$ (cf.~Theorem~\ref{theorem:continuity_of_eigenvalues} and Remark~\ref{remark:no_Lip}) implies $\tilde \lambda_2(t_n^{-1}) \to \tilde \lambda_2(0) > 0$, which is a contradiction.
\end{proof}

\subsection{Stability in the case $B\equiv Id$}\label{section:stab1}

In this subsection we consider the case that the operators $B(t)$ in $\CP$ equal the identity operator for all $t\in [0,\infty)$.
For the system \eqref{diff-sys}, \eqref{diff-sys-boundary} from the introduction this means that
\[
	\mu_j(t) = c_j(t) \quad (t \ge 0, j=1,\dots,m).
\]
Thus we have conservation of mass, i.e.~$\scp {u(t)} {e_1}_H$ is constant for each solution $u$ provided that $\scp{F(t)}{e_1}_H \equiv 0$.

\begin{proposition}\label{proposition:conservation_of_mass}
	Let  $u^0\in V$, $F \in L^2_{loc}([0, \infty);H)$ with $\scp{F(t)}{e_1}=0$ for  a.e.~$t\ge 0$, and let $u$ be the solution of $\CP$. 
Then $$\scp {u(t)} {e_1}_H = \scp {u^0} {e_1}_H =
	 \frac{1}{\sqrt{m}} \scp{u^0}{\1}_H\quad (a.e.~t \in [0,\infty)).$$
\end{proposition}
\begin{proof}
	Let $t \ge 0$. Then
	\begin{align*}
		\scp {u(t)}{e_1}_H - \scp {u^0}{e_1}_H = \int_0^t \scp{\dot u(s)}{e_1}_H \ \d s &= \int_0^t \scp{F(s) - A(s) u(s)}{e_1}_H \ \d s\\
						    = \int_0^t \left(\scp{F(s)}{e_1}_H - \fra(s,u(s),e_1) \right)\ \d s &= 0. {\tag*{\qedhere}}
	\end{align*}
\end{proof}

We are now in the position to prove  that in  case $B\equiv Id$ and $F \in L^2(0,T;H)$ the solution $u$ of $\CP$ consists of a part that is constant
on the graph and a part converging exponentially to $0$.
In particular $u(t)$ converges as $t \to \infty$.

\begin{theorem}\label{thm:stability}
	Let  $u^0 \in V$ and $F \in L^2(0,\infty; H)$ such that
	\[
		F_\infty :=\lim_{t\to\infty} \int_0^t \scp{F(s)}{e_1} \ \d{s}
	\]
	exists. Set $w(t) := \scp{u_0}{e_1}_H e_1 +  \int_0^t \scp{F(s)}{e_1}_H \, \d{s} \,e_1$.
	Then the solution $u$ of $\CP$ can be decomposed as $$u(t) = w(t) + \tilde u(t) \quad (t \ge 0)$$ where
	$w(t)$ converges to $\scp{u_0}{e_1}_H e_1 + F_\infty e_1$ as $t \to \infty$ 
	and $\tilde u(t) := u(t)-w(t)$ is exponentially stable.
\end{theorem}
\begin{proof}
	Note that $w(t) \in D(A(t))$ and $A(t)w(t)= 0$ for all $t \ge 0$, thus
	\[
		\dot {\tilde u} + A \tilde u = F - \dot w = F - \scp{F}{e_1}_H e_1 =: \tilde F, \quad \tilde u(0) = u_0 -\scp{u_0}{e_1}e_1.
	\]
	By Proposition~\ref{proposition:conservation_of_mass} we have $\scp{\tilde u(\cdot)}{e_1}_H \equiv 0$. 
	This identity together with
	Theorem~\ref{theorem:spectral_representation} shows that
	\[
		\fra(t,\tilde u(t),\tilde u(t)) \ge \lambda_2(t) \norm{\tilde u(t)}_H^2 \quad (a.e.~t \in [0,\infty)).
	\]
	Let $\varepsilon >0$, then by the product rule, Young's inequality and the above estimate
	\begin{align*}
		&\norm{\tilde u(t)}_H^2 - \norm{\tilde u(0)}_H^2\\
		&\quad= \int_0^t 2 \scp{\dot {\tilde u} (s)}{\tilde u(s)}_H \ \d s\\
		&\quad= 2 \int_0^t \scp{\tilde F(s)- A(s)\tilde u(s)}{\tilde u(s)}_H \ \d s\\
		&\quad= 2 \int_0^t \scp{\tilde F(s)}{\tilde u(s)}_H \ \d s - 2 \int_0^t \fra(s, \tilde u(s), \tilde u(s)) \ \d s\\
		&\quad\le \tfrac 1 {2\varepsilon} \int_0^t \norm{\tilde F(s)}_H^2 \ \d s + 2 \varepsilon \int_0^t \norm{\tilde u(s)}_H^2 \ \d s - 2 \int_0^t \lambda_2(s) \norm{\tilde u(s)}_H^2 \ \d s\\
		&\quad= \tfrac 1 {2\varepsilon} \int_0^t \norm{\tilde F(s)}_H^2 \ \d s -2 \int_0^t (\lambda_2(s)-\varepsilon) \norm{\tilde u(s)}_H^2 \ \d s.
	\end{align*}
	By Gronwall's inequality we obtain
	\begin{align*}
		\norm{\tilde u(t)}_H 
			&\le \Big(\norm{u^0-\tilde u^0}_H^2 + \tfrac 1 {2\varepsilon} \int_0^t \norm{\tilde F(s)}_H^2 \ \d s \Big)^{1/2} \exp\Big(\!- \int_0^t (\lambda_2(s)-\varepsilon) \ \d s \Big)\\
			&\le \Big(\norm{u^0-\tilde u^0}_H^2 + \tfrac 1 {2\varepsilon} \int_0^t \norm{F(s)}_H^2 \ \d s \Big)^{1/2} \exp\Big(\!- t (\underline\lambda_2-\varepsilon) \Big)
	\end{align*}
	for all $t \ge 0$. This shows exponential stability of $\tilde u$ if we choose $\varepsilon < \underline \lambda_2$.
\end{proof}
\begin{remark}
	If the function $F$ is only in $L^2_{loc}([0, \infty);H)$ with 
	$$\norm{F}_{L^2(0,t;H)} \le C e^{t(\underline\lambda_2 - \delta)} \quad (t \ge 0)$$
	or even
	$$\norm{F}_{L^2(0,t;H)} \le C \exp\Big( \int_0^t (\lambda_2(s)-\delta) \ \d s \Big) \quad (t \ge 0)$$
	for some $\delta >0$ and $C \ge 0$, 
	then the assertion of Theorem~\ref{thm:stability} still holds.
\end{remark}

Observe that the  speed of the convergence of the exponentially stable part depends on the value of $\underline\lambda_2$ which is strongly related to the structure of the network (for a discussion on these connections we refer to \cite[Section 5]{KMS07}).

\subsection{Stability in the case $\dot b_i \le 0$}\label{section:stab2}
In this subsection we consider the case that the functions $b_i$ are Lipschitz continuous, i.e.\ $b_i \in W^{1,\infty}(0,\infty)$, 
and that $\dot b_i(t) \le 0$ for a.e.\ $t \in [0, \infty)$.
Again we denote by $u$ the unique solution of $\CP$.
Conservation of mass does not hold in this situation, i.e.\ $\scp {u(\cdot)} {e_1}_H$ may vary in time.
Thus we cannot expect the same result as before, but we can still show a nice asymptotic bahaviour of the solution. 

\begin{theorem}\label{theorem:exp_stability2}
	Let $u^0 \in V$ and  $F \in L^2(0,\infty; H)$ with $\scp{F(t)}{e_1}_H = 0$ for a.e.\ $t \ge 0$.
	Then the solution $u$ of $\CP$ can be decomposed as $$u=\tilde u_1 +\tilde u, $$	
	where $\tilde u_1(t):= \scp{u(t)}{e_1}_H e_1 = \frac{1}{m}\scp{u(t)}{\1}_H \1$ converges to an equilibrium in $H$ as $t \to \infty$ 
	and  $\tilde u(\cdot) := u(\cdot) - \tilde u_1(\cdot)$ behaves exponentially stable.
\end{theorem}
For the proof we need the following lemma.
\begin{lemma}\label{lemma:bounded_almost_monotone}
	Let $f\colon [0,\infty) \to [0,\infty)$, $g \in L^1(0,\infty)$ such that for $s \le t$
	\[
		f(t)-f(s) \le \int_s^t g(r) \ \d r.
	\]
	Then the limit $\lim_{t\to\infty}f(t)$ exists.
\end{lemma}
\begin{proof}
	Define $a_n = \inf\{ f(t): t \ge n \}$.
	Then 
	\[
		a_1 \le a_2 \le \dots \le f(0) + \int_0^\infty \abs{g(r)}\ \d r,
	\] 
	hence $a:= \lim_{n\to\infty}a_n$ exists.
	Let $t_n \ge n$ such that $f(t_n) \le a_n +\frac 1 n$.
	Then for $t \ge t_n$ we have 
	\[
		a_n \le f(t) \le f(t_n) + \int_{t_n}^t g(r)\ \d r \le a_n +\frac 1 n + \int_n^\infty \abs{g(r)}\ \d r.
	\]
	Hence $\lim_{t\to\infty} f(t) = a$.
\end{proof}

\begin{proof}[Proof of Theorem \ref{theorem:exp_stability2}]
	Theorem~\ref{theorem:spectral_representation} shows that
	\[
		\fra(t,\tilde u(t),\tilde u(t)) \ge \lambda_2(t) \norm{\tilde u(t)}_H^2
	\]
	for a.e.\ $t\in [0,\infty)$. Let $\varepsilon >0$. Now by \eqref{eq:B_bounds}, the product rule, Young's inequality and the above estimate
	\begin{equation}\label{eq:norm_estimate2}\begin{split}	
		&\beta^2 \norm{\tilde u(t)}_H^2 -\beta^{-2} \norm{u(0)}_H^2 \le \beta^2 \norm{ u(t)}_H^2 - \beta^{-2}\norm{u(0)}_H^2\\
		&\quad\le \norm{B^{1/2}(t) u(t)}_H^2 - \norm{B^{1/2}(0) u(0)}_H^2\\
		&\quad= 2 \int_0^t  \scp{B(r)\dot u(r)}{u(r)}_H \ \d r + \int_0^t \scp{\dot B(r)u(r)}{u(r)}_H \ \d r\\
		&\quad\le 2 \int_0^t \scp{F(r)- A(r) u(r)}{u(r)}_H \ \d r\\
		&\quad= 2 \int_0^t \scp{F(r)}{u(r)}_H \ \d r - 2 \int_0^t \fra(r, u(r), u(r)) \ \d r\\
		&\quad= 2 \int_0^t \scp{F(r)}{\tilde u(r)}_H \ \d r - 2 \int_0^t \fra(r, \tilde u(r), \tilde u(r)) \ \d r\\
		&\quad\le \tfrac 1 {2\varepsilon} \int_0^t \norm{F(r)}_H^2 \ \d r + 2 \varepsilon \int_0^t \norm{\tilde u(r)}_H^2 \ \d r - 2 \int_0^t \lambda_2(r) \norm{\tilde u(r)}_H^2 \ \d r\\
		&\quad= \tfrac 1 {2\varepsilon} \int_0^t \norm{F(r)}_H^2 \ \d r -2 \int_0^t (\lambda_2(r)-\varepsilon) \norm{\tilde u(r)}_H^2 \ \d r.
	\end{split}\end{equation}
	By Gronwall's inequality we obtain
	\begin{equation}\label{eq:exp_stability2}
		\begin{split}
		\norm{\tilde u(t)}_H 
			&\le \Big(\tfrac 1{\beta^{2}} \norm{u^0}_H^2 +  \tfrac 1{2\varepsilon {\beta}}  \int_0^t \norm{F(r)}_H^2 \ \d r \Big)^{1/2} \exp\Big( \tfrac {-1}{{ \beta} }\int_0^t (\lambda_2(r)-\varepsilon) \ \d r \Big)\\
			&\le \Big(\tfrac1{\beta^2} \norm{u^0}_H^2 +  \tfrac 1{2\varepsilon  {\beta}} \int_0^t \norm{F(r)}_H^2 \ \d r \Big)^{1/2} \exp\Big( \tfrac{-t}{\beta}(\underline\lambda_2-\varepsilon) \Big)
		\end{split}
	\end{equation}
	for all $t \ge 0$. This shows exponential stability of $\tilde u$ if we choose $\varepsilon < \underline \lambda_2$.
	
	Let $0\le s \le t$. If we replace the lower integral limit $0$  in $\eqref{eq:norm_estimate2}$ by $s$ we obtain
	\begin{equation*}\begin{split}	
		&\norm{B^{1/2}(t) u(t)}_H^2 - \norm{B^{1/2}(s) u(s)}_H^2\\
		&\quad\le \tfrac 1 {2\varepsilon} \int_s^t \norm{F(r)}_H^2 \ \d r - 2 \int_s^t (\lambda_2(r)-\varepsilon) \norm{\tilde u(r)}_H^2 \ \d r.
	\end{split}\end{equation*}
	We insert $\eqref{eq:exp_stability2}$ in the above estimate and conclude that $\norm{B^{1/2}(t) u(t)}_H^2$ converges for $t \to \infty$ by Lemma~\ref{lemma:bounded_almost_monotone}.
	Since the functions $b_i(t)$ are monotone and bounded they converge as $t\to\infty$, thus also $B^{1/2}(t)$ converges as $t\to\infty$.
	Now 
	\begin{align*}
		\scp{u(t)}{e_1}_H^2 \norm{B^{1/2}(t) e_1}_H^2 + \norm{B^{1/2}(t) \tilde u(t)}_H^2 =  \norm{B^{1/2}(t) u(t)}_H^2
	\end{align*}
	together with $\eqref{eq:exp_stability2}$ shows that $\scp{u(t)}{e_1}_H^2$ converges as $t\to\infty$.
	Finally, since the function $\scp{u(\cdot)}{e_1}_H$ is continuous also $\scp{u(t)}{e_1}_H$ converges as $t\to\infty$,
	thus $\tilde u_1(t)=\scp{u(t)}{e_1}_H e_1$ converges in $H$ as $t\to\infty$.
\end{proof}

\subsection{Stability in the case $\dot b_i(t) \le 2 c b_i(t)$ for a $c < \underline\lambda_2$}\label{section:stab3}

In this subsection we consider the case that the functions $b_i$ are Lipschitz continuous, i.e.\ $b_i \in W^{1,\infty}(0,\infty)$, 
and that $\dot b_i(t) \le 2 c b_i(t)$ for some $c < \underline\lambda_2$ and almost all $t \in [0, \infty)$.
Again we denote by $u$ the unique solution of $\CP$.
In this situation we can still decompose the solution as $u=\tilde u_1 +\tilde u $ and show that the function $\tilde u$ is exponentially stable
and $\tilde u_1$ is a multiple of the first eigenfunction.
But we do not know further properties of $\tilde u_1$ like convergence as $t\to\infty$, even boundedness is not clear.

\begin{theorem}
	Let $u^0 \in V$ and $F \in L^2(0,\infty; H)$ with $\scp{F(t)}{e_1}_H = 0$ for a.e.\ $t \ge 0$. Then the solution $u$ of $\CP$ can be decomposed as $$u=\tilde u_1 +\tilde u, $$	
	where $\tilde u_1(t):= \frac{1}{m}\scp{u(t)}{\1}_H \1$ 	and  $\tilde u(\cdot)$ is exponentially stable.
\end{theorem}
\begin{proof}
	We set $v(\cdot):= e^{-c \cdot} u(\cdot)$ and $\tilde F(\cdot) = e^{- c \cdot} F(\cdot)$, then since $u$ is a solution of $\CP$
	\begin{align*}
		B(t) \dot v(t) &= e^{-c t} (B(t) \dot u(t)-c B(t) u(t))\\
			&= \tilde F(t) - A(t) v(t) - c B(t) v(t).
	\end{align*}
	Let $\tilde v(\cdot) := v(\cdot)-\scp{v(\cdot)}{e_1}_H e_1$.
	Theorem~\ref{theorem:spectral_representation} shows that
	\[
		\fra(t,\tilde v(t),\tilde v(t)) \ge \lambda_2(t) \norm{\tilde v(t)}_H^2
	\]
	for a.e.\ $t\in [0,\infty)$. Let $\varepsilon >0$. Now by \eqref{eq:B_bounds}, the product rule and Young's inequality, we obtain for $t \ge 0$ the estimate
	\begin{align*}
		&\beta^2 \norm{\tilde v(t)}_H^2 -\beta^{-2} \norm{u(0)}_H^2 \le \beta^2 \norm{v(t)}_H^2 - \beta^{-2}\norm{v(0)}_H^2\\
		&\quad\le \norm{B^{1/2}(t) v(t)}_H^2 - \norm{B^{1/2}(0) v(0)}_H^2\\
		&\quad= 2 \int_0^t  \scp{B(s)\dot v(s)}{v(s)}_H \ \d s + \int_0^t \scp{\dot B(s)v(s)}{v(s)}_H \ \d s\\
		&\quad= 2 \int_0^t \scp{\tilde F(s)- A(s) v(s)}{v(s)}_H\ \d s + \int_0^t \scp{(\dot B(s) - 2cB(s)) v(s)}{v(s)}_H\ \d s\\
		&\quad\le 2 \int_0^t \scp{\tilde F(s)}{v(s)}_H \ \d s - 2 \int_0^t \fra(s, v(s), v(s)) \ \d s\\
		&\quad= 2 \int_0^t \scp{\tilde F(s)}{\tilde v(s)}_H \ \d s - 2 \int_0^t \fra(s, \tilde v(s), \tilde v(s)) \ \d s\\
		&\quad\le \tfrac 1 {2\varepsilon} \int_0^t \norm{\tilde F(s)}_H^2 \ \d s + 2 \varepsilon \int_0^t \norm{\tilde v(s)}_H^2 \ \d s - 2 \int_0^t \lambda_2(s) \norm{\tilde v(s)}_H^2 \ \d s\\
		&\quad= \tfrac 1 {2\varepsilon} \int_0^t \norm{\tilde F(s)}_H^2 \ \d s -2 \int_0^t (\lambda_2(s)-\varepsilon) \norm{\tilde v(s)}_H^2 \ \d s.
	\end{align*}
	By Gronwall's inequality we obtain
	\begin{align*}
		\norm{\tilde v(t)}_H 
			&\le \Big(\tfrac{1}{ \beta^2}\norm{u^0}_H^2 + \tfrac 1 {2\varepsilon {\beta}} \int_0^t \norm{\tilde F(s)}_H^2 \ \d s \Big)^{1/2} \exp\Big(\!\tfrac {-1}{{ \beta} }\int_0^t (\lambda_2(s)-\varepsilon) \ \d s \Big)\\
			&\le \Big(\tfrac{1}{ \beta^2}\norm{u^0}_H^2 + \tfrac 1 {2\varepsilon {\beta}}  \int_0^t \norm{\tilde F(s)}_H^2 \ \d s \Big)^{1/2} \exp\Big(\!- \tfrac {t}{{ \beta} }(\underline\lambda_2-\varepsilon) \Big)
	\end{align*}
	for all $t \ge 0$. Thus 
	\begin{equation*}
		\norm{\tilde u(t)}_H \le \Big(\tfrac{1}{ \beta^2}\norm{u^0}_H^2 + \tfrac 1 {2\varepsilon {\beta}}\int_0^t \norm{\tilde F(s)}_H^2 \ \d s \Big)^{1/2} \exp\Big(\!- \tfrac {t}{{ \beta} } (\underline\lambda_2 - c -\varepsilon) \Big)
	\end{equation*}
	for all $t \ge 0$. This shows exponential stability of $\tilde u$ if we choose $\varepsilon < \underline \lambda_2 - c$.
\end{proof}
%%%%%%%%%%%%%%%%%%%%%%%%%%%%%%%%%%%%%%%%%%%%%%%%%%%%%%%%%%%%%%%%%%%%%%%%%%%%%%%%%%%%%%%%%%%%%%%%%%%%%%%%%%%%%%%%%%%%%%%%%%%%%%%%

\section{Appendix}
In this appendix, based on Kato's monograph \cite{Kat66}, we prove continuous dependency of the $k$-th eigenvalue as we need it in Section~\ref{section:stability}.
This result is folklore, but we were not able to find a reference.

Let $H$ be a separable complex Hilbert space of infinite dimension. 
Suppose $A$ is an unbounded self-adjoint operator on $H$ with $\sigma(A) \subset [\omega_0, \infty)$ 
and compact resolvent (i.e.\ $R(\omega, A)=(\omega-A)^{-1}$ is a compact operator on $H$ for some $\omega < \omega_0$).
Then by the Spectral Theorem, the space $H$ has an orthonormal basis $\{ e_k : k\in \N \}$ such that $e_k \in D(A)$, $Ae_k = \lambda_k e_k$
where $ \omega_0 \le \lambda_1 \le \dots \le \lambda_k \le \lambda_{k+1}\le \dots$ and $\lim_{k\to\infty} \lambda_k =\infty$.
As a consequence $\{ \lambda_k : k \in\N \}$ is exactly the set of all eigenvalues of $A$ and hence $\lambda_k$ is the $k$-th eigenvalue of $A$ counting multiplicity.
For  $n \in \N$ let $A_n$ be a further unbounded self-adjoint operator on $H$ with $\sigma(A_n) \subset [\omega_0, \infty)$ and compact resolvent.
We denote the $k$-th eigenvalue of the operator $A_n$ by $\lambda_k^n$.
\begin{theorem}\label{theorem:convergence_of_eigenvalues}
	Suppose $\|R(\omega, A_n) - R(\omega,A)\|_{\L(H)}\to 0$ as ${n\to\infty} $ for some $\omega < \omega_0$. Then $\lim_{n\to\infty} \lambda_k^n = \lambda_k$ for every $k\in\N$.
\end{theorem}

We show two auxiliary results before proving the theorem. 

\begin{lemma}\label{lemma:convergency_of_eigenvalues}
Let $\omega_0\le\mu_1<\dots<\mu_p<\mu_{p+1}<\dots$ be all different eigenvalues of $A$, $\mu_p$ having multiplicity $m_p \in\N$, and let $\omega_0 \le \lambda_k^n \le \lambda_{k+1}^n \le \dots$ be such that the following two conditions hold.
	\begin{enumerate}[label={\rm \alph*)}]
		\item For each $p\in\N$, $\varepsilon >0$ such that $\mu_{p-1} + \varepsilon < \mu_p < \mu_{p+1} - \varepsilon$ there exists a $n_0 \in\N$ such that
			for all $n\ge n_0$ one has
			\[
				\#\{ k : \lambda_k^n \in [\mu_p-\varepsilon, \mu_p+\varepsilon] \} = m_p.
			\]
		\item For all $k\in\N$ each limit point of $(\lambda_k^n)_{n\in\N}$ is in $\{ \mu_p : p\in\N \}$.
	\end{enumerate}
	Then $\lim_{n\to\infty} \lambda_k^n = \lambda_k$ for all $k\in\N$ where $\lambda_k =\mu_p$ for $M_{p-1}< k \le M_p$,
	$M_p = m_1+\dots + m_p$ for $p\in\N$ and $M_0=0$.
\end{lemma}
\begin{proof}
	The proof is given by induction over $p$.
	Let $p=1$ and let $\varepsilon >0$ such that $\mu_1+\varepsilon <\mu_2$.
	By a) there exists a $n_0$ such that $\lambda_{M_1}^n \le \mu_1+\varepsilon$ for all $n \ge n_0$.
	Thus for $k=1,\dots,M_1$, the sequence $(\lambda_k^n)_{n\in\N}$ is bounded and by b) it has $\mu_1$ as the only possible limit point.
	Hence $\lim_{n\to\infty}\lambda_k^n = \mu_1$ for $k=1,\dots,M_1$.

	Now let $p\in\N$, $p\ge 2$ such that the assertion of the lemma holds for $p-1$.
	Let $\varepsilon >0$ such that $\mu_{p-1}<\mu_p-\varepsilon<\mu_p<\mu_p+\varepsilon<\mu_{p+1}$.
	By the inductive hypothesis there exists $n_0\in\N$ such that for all $n\ge n_0$, $\lambda_{M_{p-1}}^n<\mu_p-\varepsilon$.
	It follows from the assumption a) that there exists $n_1 \ge n_0$ such that $\lambda_{M_p}^n<\mu_p +\varepsilon$ for all $n\ge n_1$.
	Hence the only possible limit point of $(\lambda_k^n)_{n\in\N}$ is $\mu_p$ whenever $M_{p-1}< k \le M_p$.
	Thus $\lim_{n\to\infty}\lambda_k^n = \mu_p$ for $M_{p-1}< k \le M_p$.
\end{proof}

\begin{lemma}\label{lemma:resolvent_property}
	Let $\omega < \omega_0$ such that $\|R(\omega, A_n) - R(\omega,A)\|_{\L(H)}\to 0$ as $n\to \infty$.
	If $\mu_n \in \sigma(A_n)$ for $n \in \N$ and $\mu_n \to \mu$ as $n \to \infty$, then $\mu \in \sigma(A)$.
\end{lemma}
\begin{proof}
	Since $\mu_n \in \sigma(A_n)$, exist eigenvectors $e_n$ such that $A_n e_n = \mu_n e_n$ for all $n \in \N$ and $\lVert e_n \rVert_H = 1$.
	We get
	\begin{equation}\label{eq:proof_resolvent_property}
	\begin{split}
		e_n &= R(\omega,A_n)(\omega-A_n) e_n = R(\omega,A_n)(\omega-\mu_n) e_n\\
			&= (\omega-\mu_n)(R(\omega, A_n) - R(\omega, A)) e_n +  (\omega-\mu_n)R(\omega, A) e_n
	\end{split}
	\end{equation}
	for all $n\in\N$. 
	Since $R(\omega, A)$ is compact and $\norm{e_n}_H=1$, after taking a subsequence we may assume that $R(\omega, A)e_n$ converges.
	This, the convergence of the resolvent and $\eqref{eq:proof_resolvent_property}$ imply that also $e_n$ converges to some $e$ as $n\to\infty$ and
	that $e = (\omega-\mu)R(\omega, A) e$ or $A e = \mu e$. 
	Thus $\mu \in \sigma(A)$.
\end{proof}

\begin{proof}[Proof of Theorem~\ref{theorem:convergence_of_eigenvalues}]
We apply Lemma~\ref{lemma:convergency_of_eigenvalues} to $\lambda_k, \lambda_k^n,\mu_p,m_p$ defined above.
Condition b) is satisfied by Lemma~\ref{lemma:resolvent_property}.
Next we want to prove condition a).
Let $p \in \N$ and $\varepsilon > 0$ such that $\mu_{p-1} + \varepsilon < \mu_p < \mu_{p+1} - \varepsilon$.
It follows from b) that there exists $n_0\in\N$ such that $\mu_{p-1} + \varepsilon, \mu_{p+1} - \varepsilon \in \rho(A)$ for all $n \ge n_0$.
For $n \ge n_0$ we denote by
\begin{equation}
	P_n:= \frac 1 {2\pi i}\int_{\Gamma} R(\omega,A_n) \ \d\omega
\end{equation}
the spectral projection associated with $(\mu_{p-1} + \varepsilon, \mu_{p+1} - \varepsilon) \cap \sigma(A)$.
Here $\Gamma$ is the circle with center $\frac 1 2 (\mu_{p-1} + \mu_{p+1})$ and radius $r = \frac 1 2 (\mu_{p+1} - \mu_{p-1}) - \varepsilon$.
Similarly, let
\begin{equation}
	P:= \frac 1 {2\pi i} \int_\Gamma R(\omega,A) \ \d\omega.
\end{equation}
Then $\norm{P_n-P}_{\L(H)}\to 0$ as $n \to \infty$.
Hence there exists $n_1 \ge n_0$ such that $\norm{P_n-P}_{\L(H)} < 1$ for all $n \ge n_1$.
Consequently $P_n$ and $P$ have the same dimension \cite[I.\ Theorem 6.32]{Kat66}.
This shows that condition a) holds.
Now the assertion of Theorem~\ref{theorem:convergence_of_eigenvalues} follows from Lemma~\ref{lemma:convergency_of_eigenvalues}.
\end{proof}

\subsection*{Acknowledgment}

The authors are grateful to Delio Mugnolo for many stimulating discussions. Dominik Dier  is a member of the DFG Graduate School 1100: Modelling, Analysis and Simulation in Economathematics. Marjeta Kramar Fijav\v{z} acknowledges a research stay at the  University of Ulm financed by the Faculty of Mathematics and Economics and the German Academic Exchange Service (DAAD).

%%%%%%%%%%%%%%%%%%%%%%%%%%%%%%%%%%%%%%%%%%%%%%%%%%%%%%%%%%%%%%%%%%%%%%%%%%%%%%%%%%%%%%%%%%%%%%%%%%%%%%%%%%%%%%%%%%%%%%%%%%%%%%%%

\noindent
\emph{Wolfgang Arendt}, \emph{Dominik Dier},\\
Institute of Applied Analysis, 
University of Ulm, 89069 Ulm, Germany,\\
\texttt{wolfgang.arendt@uni-ulm.de}, \texttt{dominik.dier@uni-ulm.de}

\medskip\noindent
\emph{Marjeta Kramar Fijav\v{z}}, \\University of Ljubljana, Faculty of Civil and Geodetic Engineering, Jamova 2, SI-1000 Ljubljana, Slovenia \\
and\\
Institute of Mathematics, Physics, and Mechanics,
Jadranska 19, SI-1000 Ljubljana, Slovenia\\
\texttt{marjeta.kramar@fgg.uni-lj.si}%, +38614768546(phone), +38614250681(fax).

\begin{thebibliography}{999}

\bibitem[ABN01]{ABN01} F.~Ali Mehmeti, J.~von Below, and S.~Nicaise (eds.), \emph{{PDE}'s on {M}ultistructures} (Proceedings Luminy/France 1999), Lect.~Notes Pure Appl.~Math.~{\bf 219}, Marcel Dekker, 2001.

\bibitem[ADLO12]{ADLO12} W.\ Arendt, D.\ Dier, H.\ Laasri, E.\ M.\ Ouhabaz,
\emph{Maximal Regularity for Evolution Equations Governed by Non-Autonomous Forms}.
Preprint. {\tt arXiv:1303.1166}

\bibitem[ADO13]{ADO13} W.\ Arendt, D.\ Dier, E.\ M.\ Ouhabaz, 
\emph{Invariance of Convex Sets for Non-autonomous Evolution Equations Governed by Forms}.
Preprint. {\tt arXiv:1303.1167}

\bibitem[AU10]{AU10} W.\ Arendt, K. Urban, \emph{Partielle Differenzialgleichungen}.
Spektrum Akademischer Verlag, Heidelberg, 2010.

\bibitem[BDK13]{BDK13} F.~Bayazit, B.~Dorn, and M.~Kramar Fijav\v{z}, \emph{Asymptotic periodicity of flows  in time-depending networks}. Preprint.  {\tt arXiv:1302.4196}

\bibitem[Be88]{Be88} J.~von Below, \emph{Classical solvability of linear parabolic equations on networks}. J.~Diff.~Equations~{\bf 72} (1988),  316--337.

\bibitem[Bre11]{Bre11}
H.\ Br\'ezis,
\emph{Functional Analysis, Sobolev Spaces and Partial Differential Equations}.
Springer, Berlin 2011.

\bibitem[Kat66]{Kat66} T.~Kato,
\emph{Perturbation Theory for Linear Operators}.
Die Grund\-lehren der mathematischen Wissenschaften {\bf } 132, Springer-Verlag, New York 1966. 

\bibitem[KMS07]{KMS07} M.~Kramar Fijav\v{z}, D.~Mugnolo, and E.~Sikolya, \emph{Variational and semigroup methods for waves and diffusion in networks}. Appl.~Math.~Optim.~{\bf 55} (2007), 219--240.

\bibitem[LLS94]{LLS94} J.~E.~Lagnese, G.~Leugering, and E.J.P.G.~Schmidt, \emph{Modeling, Analysis, and Control of Dynamic Elastic Multi-Link Structures}. Syst.~Contr.: Found.~Appl., Birkh\"{a}user Verlag, 1994.


\bibitem[Ouh05]{Ouh05} E.\ M.\ Ouhabaz, \textit{Analysis of Heat Equations on Domains}. London Math.\ Soc.\ Monographs,
Princeton Univ.\ Press, 2005. 


\bibitem[Tan79]{Tan79} H.\ Tanabe, \textit{Equations of Evolution.} Pitman 1979.

\end{thebibliography}
\end{document}